\documentclass[11pt]{amsart}

\usepackage{amsmath}
\usepackage{fullpage}
\usepackage{xspace}
\usepackage[psamsfonts]{amssymb}
\usepackage[latin1]{inputenc}
\usepackage{graphicx,color}
\usepackage{hyperref}
\usepackage{graphicx}

\usepackage{amsmath}%
\usepackage{amsthm}%
\usepackage{amscd}
\usepackage{amsfonts}%
\usepackage{amssymb}%
\usepackage{graphicx}
\usepackage{stmaryrd}
\usepackage{mathrsfs}
\usepackage{tikz}
\usetikzlibrary{matrix,arrows}
\usepackage{tikz-cd}
\usepackage[all]{xy}
\usepackage{bbm} 
\usepackage{longtable}

\newtheorem{theorem}{Theorem}[section]

\newtheorem{corollary}[theorem]{Corollary}

\newtheorem{definition}{Definition}
\newtheorem{example}{Example}

\newtheorem{lemma}[theorem]{Lemma}

\theoremstyle{remark}
\newtheorem{remark}[theorem]{Remark}

\numberwithin{equation}{section}

\newcommand{\Pro}{\mathbb{P}}
\newcommand{\Q}{\mathbb{Q}}

\newcommand{\Z}{\mathbb{Z}}

\newcommand{\Gal}{\mathrm{Gal}}

\newcommand{\hhat}{\hat{h}}

\newcommand{\kbar}{\overline{k}}
\newcommand{\kleqnu}{\overline{k}^{\leq \nu}}

\newcommand\rquot[2]{
  \mathchoice
  {
    \text{\raise0.5ex\hbox{$#1$}\big/\lower0.5ex\hbox{$#2$}}%
  }
  {
    #1\,/\,#2
  }
  {
    #1\,/\,#2
  }
  {
    #1\,/\,#2
  }
}

\newcommand\lrquot[3]{
  \mathchoice
  {
    \text{\lower0.5ex\hbox{$#1$}\big\backslash\raise0.5ex\hbox{$#2$\!}\big/
      \lower0.5ex\hbox{\!\!$#3$}}%
  }
  {
    #1\,\backslash\,#2\,/\,#3
  }
  {
    #1\,\backslash\,#2\,/\,#3
  }
  {
    #1\,\backslash\,#2\,/\,#3
  }
}

\newcommand\lquot[2]{
  \mathchoice
  {
    \text{\lower0.5ex\hbox{$#1$}\big\backslash\raise0.5ex\hbox{$#2$}}%
  }
  {
    #1\,\backslash\,#2
  }
  {
    #1\,\backslash\,#2
  }
  {
    #1\,\backslash\,#2
  }
}

\usepackage{fontawesome}

  \DeclareFontFamily{U}{wncy}{}
    \DeclareFontShape{U}{wncy}{m}{n}{<->wncyr10}{}
    \DeclareSymbolFont{mcy}{U}{wncy}{m}{n}
    \DeclareMathSymbol{\Sha}{\mathord}{mcy}{"58}

\begin{document}
\title{Counting algebraic points of bounded degree on curves}
\author{Matias Alvarado}
\address{ Departamento de Matem\'aticas,
Pontificia Universidad Cat\'olica de Chile.
Facultad de Matem\'aticas,
4860 Av.\ Vicu\~na Mackenna,
Macul, RM, Chile}
\email[M. Alvarado]{matias.alvarado.torres1@gmail.com }
\date{\today}
\maketitle
\begin{abstract}

Let $X$ be a smooth projective curve over a number field $k$. Let $f\colon X \to \Pro^1$ be a nonconstant morphism over $k$ that realizes the gonality of $X$.
In this article, we study the growth rate of $\left\{P\in X\left(\overline{k} \right)\left| [k(x):k]=\nu , k(x)=k(f(x)), h(x)\leq T \right.\right\}$ for a fixed $\nu$.
\end{abstract}

\setcounter{tocdepth}{1}


\section{Introduction}

Let $k$ be a number field, and $X$ be a smooth projective curve over $k$. Faltings theorem establishes that the set of rational points $X(L)$ is finite for any finite extension $L/k$ if $g(X)>1$. In this article, we study algebraic points that belong to finite extensions of a fixed degree.
For example, if $X$ is a hyperelliptic curve given by $y^2=f(x)$ (of genus greater than 1) defined over a number field $k$, then we can ask about quadratic points. We observe that there are infinitely many points of degree 2. Namely, any point of the form $(x,\sqrt{f(x)})\in X(\overline{k})$ (for $x\in k$) is a point of degree 1 or 2 over $k$, but by Faltings theorem, there are only finitely many $k$-rational points. Then there are infinitely many quadratic points. In previous example we observe that these quadratic points are preimages of $k$-rational of $\Pro^1$ via the natural $2:1$ covering  $X\to \Pro^1$ given by $(x,y)\mapsto x$. 

If $\nu$ is a natural number, and $f\colon X \to \Pro^1$ is a dominant morphism, then we will focus on counting points $x\in X(\overline{k})$ such that $[k(x):k]=\nu$ and $k\left(f(x)\right)=k(x)$ ordered by height.
Some properties of these points have been studied by Vojta in \cite{vojta1989arithmetic}. In particular Vojta proved that under the presence of a dominant morphism $f \colon X\to \Pro^1$, there is numerical condition (which involves the genus of $X$, $\nu$ and $\deg f$) that ensures the finiteness of the set $\{x\in X\left(\overline{k}\right):[k(x):k]\leq \nu \text{ and } k(f(x))=k(x)\}$. 
In \cite{Songtucker01}, Song and Tucker study same problem but with morphisms to general curves $C'$ instead of $\Pro^1.$

Before stating our main result, we introduce the notation $X\left(\overline{k}^{\leq \nu}\right)$ to denote the set of algebraic point $x\in X(\overline{k})$ such that $[k(x):k]\leq \nu$. Similarly, we denote by $X( \overline{k}^{=\nu})$ the set of points $x\in X\left(\overline{k}\right)$ such that $[k(x):k]=\nu.$

The main result of this article is the following.

\begin{theorem}\label{mainthm} Let $X$ be a smooth projective over $k$ of odd gonality $\gamma$. Let $f \colon X \to \Pro^1$ be a dominant morphism of degree $\gamma$. Let $\nu$ be a prime number such that $\nu \leq \gamma(X)$.  Then, there is an integer $\rho\geq 0$, such that
    $$\#\left\{P \in X\left(k^{=\nu}\right):h(P)\leq T \ \& \ k(f(x))=k(x) \right\}\asymp T^{\rho/2}.$$
\end{theorem}
The strategy we use to prove Theorem \ref{mainthm} is studying the symmetric power of $X$ and its image in the jacobian of $X.$ which allows us to relate algebraic points of bounded degree on curves to rational points in the corresponding jacobian variety. Furthermore, we establish a relation between the heights of algebraic points on $X$ and N\'eron-Tate height of their images in the jacobian.

\subsection{Overview of the article}In section \ref{preliminaries}, we introduce some results due to Vojta on algebraic points of bounded degree. In addition, we present the strategy to study algebraic points. At the end of this section, we relate the height of algebraic points to the N\'eron-Tate height of points on the jacobian of the curve $X.$

In Section\ref{proofofmainthm}, we give the proof of \ref{mainthm}.
In Section \ref{examples} we exhibit some examples of Theorem \ref{mainthm}.

\section{Preliminaries}\label{preliminaries} 

We begin this section by introducing the notion of $f$-rigid points.
\begin{definition}
    Let $f \colon X \to \Pro^1$ be a nonconstant morphism over $k$. We say that $x\in X\left( \overline{k}\right)$ is a $f$-rigid point if $k(x)=k(f(x))$.
\end{definition}
\begin{remark}
    $f$-rigid points are exactly the points involved in Theorem \ref{mainthm}.
\end{remark}
\begin{remark}
If $X$ is the hyperelliptic curve given by $y^2=g(x)$, and $f \colon X \to \Pro^1$ the natural 2:1 covering, then the points $(x,\sqrt{f(x)})\in X(\overline{k})$ are no $f$-rigid in general (except for the values $x$ for which $f(x)$ is a square.)
\end{remark}
As we commented in the introduction, Vojta in \cite{vojta1989arithmetic} gives a numerical condition to ensure the finiteness of $f$-rigid algebraic points of bounded degree. The Vojta's result is the following.

\begin{theorem}
Let $X$ be a nice curve over $k$ of genus $g$. Let $f \colon X \to \Pro^1$ be a dominant morphism, and $\nu\geq 1$ an integer number. Suppose 
    $$g-1>(\deg f)(\nu-1).$$
    Then the set
    $$\{P\in X(\kbar):[k(P):k]\leq \nu, \text{ and }  k(f(P))=k(P)\}$$
    is finite.
\end{theorem}

Vojta proved this theorem under the hypothesis \ref{vojtainequality}, which involves the genus of the curve, the height of the point $x$, and the arithmetic discriminant $d_a(x).$
    \begin{equation}\label{vojtainequality}
    (2g-2)h(x)\leq(1+\varepsilon)d_a(x)+O(1).
\end{equation}



To study algebraic points of bounded degree, we use the following strategy. Let $X^{(\nu)}$ be the $\nu$-th symmetric power of $X$ that is defined as $X^{(\nu)}=X^{\nu}/S_{\nu}$, where the symmetric group $S_{\nu}$ acts on $X^{\nu}$ by permuting the coordinates. There is a map $X\left(\kbar^{\ =\nu} \right)\to X^{(\nu)}(k)$ which sends an algebraic point $x\in X(\overline{k})$ of degree $\nu$ to its complete Galois orbit $\left(\sigma_1(x),...,\sigma_\nu(x)\right)\in X^{(\nu)}(k)$

Let $J_X$ be the jacobian variety of $X$, then there is a morphism $\Phi_\nu \colon X^{(\nu)}\to J_X$ given by
$$\{x_1,...,x_\nu\}\mapsto \left[(2g-2)\sum x_i-\nu K_X\right],$$
where $K_X$ is the canonical divisor of $X$. We denote by $W_\nu$ the image of $X^{(\nu)}$ in $J_X$ via this morphism.
Composing the previous two maps, we get a function $X\left(\kbar^{\ =\nu}\right) \to J_X(k)$, which allows us to study algebraic points in $X$ via rational points of $J_X.$ The image of a point $p\in X\left(\kbar^{\ =\nu}\right)$ will be denoted by $D_p$.

Now, we establish a deeper relation between $p$ and $D_p$. In fact, we will compare $h(p)$ to the N\'eron-Tate height $\hhat_{J_X}(D_p)$. For this, we follow \cite{Songtucker01}. Note that the definition of $D_p$ in \cite{Songtucker01} is not the same as we defined, but differs only by a multiple. The reason for this discrepancy is the choice of morphism from $X^{(\nu)}$ to $J_X.$

Following equation (1.0.4) in \cite{Songtucker} and considering the appropriate renormalization we get

\begin{equation}\label{songtucker}
    \dfrac{D_p^2}{(2g-2)^2\nu}=d_a(p)-h_{K_X}(p)-2\nu h(p)+O(\nu),
\end{equation}
and at the same way by equation (1.0.5) in \cite{Songtucker01}
\begin{equation}\label{arithmetichodgeindex}
D_p^2=-2\hhat_{J_X}(D_p).
\end{equation}

By equalities \ref{songtucker} and \ref{arithmetichodgeindex}, we get

\begin{equation}\label{heightandnerontate}
    2(\nu+g-1)h(p)=d_a(p)+2\dfrac{ \hhat_{J_X}(D_p)}{(2g-2)^2\nu}+O(\nu).
\end{equation}

On the other hand, using the properties of the arithmetic discriminants due to Vojta \cite{vojta1989arithmetic}, we bound $d_a(p)$ as follows 
\begin{equation}\label{arithmeticdiscriminant}
    d_a(p)\leq d_a(f(p))=2(\nu-1)h_{\Pro^1}(f(p))+O(1)=2(\nu-1)\delta h(p)+O(1).    
\end{equation}

The inequality holds by Lemma1 3.4 (e) in \cite{vojta1989arithmetic}. The first equality holds by Lemma 3.4 (d) in \cite{vojta1989arithmetic}. The last equality follows by the functoriality of the Weil height machine.

In this way, we have the following lemma, which gives the first relation between the height $h(p)$ and $\hhat_{J_X}(D_p)$.
\begin{lemma}\label{ineq1} Let $g,\delta$ and $\nu$ as below, then
    $$\left(g-(\delta-1)(\nu-1)\right)h(p)\leq \dfrac{ \hhat_{J_X}(D_p)}{(2g-2)^2\nu}+O(\nu).$$
\end{lemma}
\begin{proof}
Replacing the inequality \ref{arithmeticdiscriminant} $d_a(p)\leq 2(\nu-1)\delta h(p)+O(1)$ in \ref{heightandnerontate} we get
$$(2\nu+2g-2)h(p)-2(\nu-1)\delta h(p)\leq 2\dfrac{\hhat(D_p)}{(2g-2)^2\nu}+O(\nu),$$
and we conclude directly 
$$\left(g-(\delta-1)(\nu-1)\right)h(p)\leq \dfrac{ \hhat_{J_X}(D_p)}{(2g-2)^2\nu}+O(\nu).$$
\end{proof}

\begin{lemma}\label{ineq2} Let $g,\delta$ and $\nu$ as below, then
    $$(\nu+g-1)h(p)\geq \dfrac{ \hhat_{J_X}(D_p)}{(2g-2)^2\nu}+O(\nu).$$
\end{lemma}
\begin{proof}
This is a consequence of equation \eqref{heightandnerontate}.
\end{proof}

\section{Proof of theorem \ref{mainthm}}\label{proofofmainthm}

\begin{proof}

First, for expository purpose we denote by $N_f(\nu,k,T)$ to
$$N_f(\nu,k,T)=\#\left\{  p \in X\left( \kbar^{ \ =\nu}\right):h(p)\leq T  \ \& \   p \ f\text{-rigid} \right\}.$$
As we saw in the previous section, the $f$-rigid points satisfy
$$ \dfrac{\nu}{\nu+g-1}\hhat_{J_X}(D_p)\leq h(p)\leq \dfrac{\nu}{g-(\delta-1)(\nu-1)}\hhat_{J_X}(D_p). $$
In other words $h(p)\asymp\hhat_{J_X}(D_p).$ The symbol $\asymp$ only depends on the fixed parameters $\delta, \nu$ and $g$. To study the growth rate of $N_f(\nu,k,T)$, we count  
the points $x\in W_\nu(k)$ such that $x$ is the image via $\Phi_\nu$ of $f$-rigid points.
As $W_\nu$ is a subvariety of an abelian variety, we can study its $k$-rational points via the Faltings theorem, which establishes that there are abelian subvarieties $A_1,...,A_m$ of $J_X$ and rational points $x_1,...,x_m \in J_X(k)$, such that 
$$W_\nu(k)=\bigcup_{i=1}^{m} \left( A_i(k)+x_i\right).$$
Each abelian subvariety $A_i$ has rank $\rho_i$. Let $\rho$ be the maximum between all the $\rho_i.$ By N\'eron theorem, there are constants $\alpha_1,\cdots \alpha_m$ such that

$$\#\left\{x\in A_i(k): \hhat(x)\leq T \right\}=\alpha_i T^{\rho_i/2}+O(T^{(\rho_i-1)/2}).$$

By N\'eron theorem and Faltings theorem, there is a constant $\alpha$ such that

\begin{equation}\label{countW}
\#\left\{x\in W_\nu(k):\hhat(x)\leq T \right\}=\alpha(A/k) T^{\rho/2}+O\left(T^{(\rho-1)/2}\right).
\end{equation}

As $h(x)\asymp \hhat(D_p)$, in order to count $f$-rigid points in $X(\overline{k}^{=\nu})$, it is enough to understand their images in $W_\nu.$

We split this problem in 3 cases. 

\noindent \textbf{Case 1:} $1\leq \nu<\gamma/2.$ 
By \cite[Prop 2]{frey}, there are at most finitely many points of degree $\nu.$ Then  $\rho=0$.

\noindent \textbf{Case 2:} $\gamma/2<\nu<\gamma.$

Let $x$ be an element in $X\left( \overline{k}^{=\nu}\right)$ such that $f(x)=y\in \Pro^1(k)$. As $f$ is defined over $k$, for any $\sigma\in \Gal(\overline{k}/k)$, $f(x^{\sigma})=y.$ Then $f^{-1}(y)=\left\{x^{\sigma_1},...,x^{\sigma_\nu},z_1,...,z_{\gamma-\nu} \right\}$, with $z_{j}\in X\left( \overline{k}^{\leq \gamma-\nu}\right)$. As $\gamma-\nu <\gamma/2$, by case 1, there are only finitely many points $z_j \in X(\overline{k}^{\gamma-\nu})$, then there are finitely many points $x\in X(\overline{k}^{=\nu})$ such that $f(x)\in \Pro^1(k)$.
We conclude that the image of $X(\overline{k}^{=\nu})$ is all $W_\nu(k)$ except for finitely many points. By equation \ref{countW}, 
$$\#\left\{x\in X(\overline{k}^{=\nu}):h(x)\leq T, \text{ and } f\text{- rigid} \right\}\asymp K_{\gamma,\nu,g}T^{\rho/2}.$$

\textbf{Case 3:} $\nu=\gamma.$
In this case, the points $x\in X\left(\overline{k}^{=\nu}\right)$ such that $k(f(x))=k$ colapse to one point via the map $X^{(\nu)}\to W_\nu\subset J_X$. Namely we can construct a morphism $\pi\colon \Pro^1\to X^{(\nu)}$ given by $\pi(t)=f^\ast(t)$. In this way the point no $f$-rigid belong to the same fiber in the morphism $x^{(\nu)}(k)\to J_X(k)$
Again by equation \ref{countW}
we conclude the existence of a constant $K_{\gamma,\nu,g}$ such that

$$\#\left\{P \in X\left(\kleqnu\right):h(P)\leq T  \right\}\asymp K_{\rho, \delta,\nu}T^{\rho/2}.$$
\end{proof}

\begin{remark} Our theorem state that the growth rate of algebraic $f$-rigid points of bounded degree growth as a polynomial. This fact is the content of the following corollary.
\end{remark}

\begin{corollary}\label{cor}
    Under the hypothsis on $\delta,\nu,g$ as before, if $\{p\in X(\overline{k}^{\leq \nu}):k(p)=k(f(p))\}$ is infinite, then
    $$N_f(\nu,k,T)\gg T^{1/2}.$$
\end{corollary}
In the theorem we take $\nu$ prime. If $x\in X(\overline{k})^{=\nu}$, then as the field $k(f(x))$ is an intermediate extension of $k(x)/k$, we have only two possibil<ities. In this way to say that a point is not $f$-rigid is equivalent to say that $k(f(x))=k.$

\section{examples}\label{examples}
In this section, we show two examples. The first example illustrates a situation in which there are no $f$-rigid points.
In the second example, we see a curve with a morphism to $\Pro^1$ with infinitely many $f$-rigid points.

\begin{example}
Let $C_{1}$ be the curve given by the plane model $y^2=x^5+1$ defined over $\Q$. $C_1$ is a hyperelliptic curve of genus $2$. Let consider $f \colon C_1 \to \Pro^1$ given by $(x,y)\mapsto x$ of degree $2$. Following \cite{mustapha}, $C_1\left(\overline{\Q}^{\leq 2} \right)=\left\{(x,\pm\sqrt{x^5+1}):x\in\Q^\ast \right\}$. Then
$$\left\{x\in C_1\left(\overline{\Q}^{= 2}\right):x \text{ is }f \text{-rigid}\right\}=\emptyset.$$
\end{example}

\begin{example}\label{example3}    
Let $X$ be the curve defined by $z^3=x^4-x^2-1$. $X$ is smooth of genus $3$ and gonality $3.$ Let $f\colon X\to \Pro^1$ be the degree $3$ morphism given by $(x,z)\mapsto x$. Let $E$ be the elliptic curve $E:y^2-y=x^3+1$ (whose label in LMFDB is $225a1$). The Mordell-Weil group $E(\Q)$ is isomorphic to $\Z$. 
If $(x,y)=(a,b)\in E(\Q)$, then $(x,z)=(\sqrt{b},a)\in X(\overline{\Q}^{\leq 2})$. By Faltings theorem most of this points are quadratic. If $(\sqrt{b},a)\in X(\overline{\Q})$ is quadratic, then $k(f(\sqrt{b},a))=k(\sqrt{b})$ which is quadratic over $\Q.$ Then we conclude that there are infinitely many $f$-rigid points.
\end{example}

\section*{Acknowledgements}

I would like to thank Hector Pasten for many discussions, suggestions, and helpful remarks. I was supported by ANID Doctorado Nacional 21200910.

\bibliographystyle{amsalpha}
\bibliography{refs.bib}

\end{document}